\documentclass[12pt]{amsart}
\usepackage{amssymb}

\newtheorem{theorem}{Theorem}
\newtheorem{proposition}{Proposition}[section]
\newtheorem{lemma}[proposition]{Lemma}
\newtheorem{corollary}[proposition]{Corollary}
\newtheorem{question}{Question}

\theoremstyle{definition}
\newtheorem*{definition}{Definition}

\newcommand{\set}[2]{\ensuremath{\{ #1 \>|\> #2 \}}}

\renewcommand{\labelenumi}{(\roman{enumi})}

\def\liebrack{\ensuremath{[\,\cdot\, , \cdot\,]}}

\DeclareMathOperator{\ad}{ad}
\DeclareMathOperator{\codim}{codim}

\begin{document}

\title{$\omega$-Lie algebras}
\author{Pasha Zusmanovich}
\address{}
\email{justpasha@gmail.com}
\date{last revised November 12, 2012}
\thanks{J. Geom. Phys. \textbf{60} (2010), 1028--1044; \texttt{arXiv:0812.0080}}

\begin{abstract}
We study a certain generalization of Lie algebras where the Jacobian
of three elements does not vanish but is equal to an expression
depending on a skew-sym\-met\-ric bilinear form.
\end{abstract}

\maketitle

\section*{Introduction}

An anticommutative algebra $L$ with multiplication $\liebrack$ over a field $K$ is 
called an \textit{$\omega$-Lie algebra} if there is a bilinear form 
$\omega: L \times L \to K$ such that
\begin{equation}\label{jacobi}
[[x,y],z] + [[z,x],y] + [[y,z],x] = \omega(x,y) z + \omega(z,x) y + \omega(y,z) x
\end{equation}
for any $x, y, z \in L$. We will refer to this identity as 
the \textit{$\omega$-Jacobi identity}.

These algebras were introduced by Nurowski in a recent interesting paper 
\cite{nurowski}\footnote[2]{
In \cite{nurowski}, the term \textit{$\omega$-deformed Lie algebra} was used.
We find this term somewhat misleading, as these are not deformations of Lie algebras 
in the usual strict sense. See, however, Question \ref{quest-deform} in 
\S \ref{section-quest}.}.
Nurowski was motivated by some physical considerations, but our treatment here 
is a purely mathematical one.

$\omega$-Lie algebras are obvious generalizations of Lie algebras, the latter 
corresponding to the case $\omega = 0$.
It follows immediately from the definition that $\omega$ is skew-sym\-met\-ric.
As noted in \cite{nurowski}, there are no $1$- and $2$-dimensional
$\omega$-Lie algebras which are not Lie algebras. 
Nurowski exhibited nontrivial examples of $3$-dimensional $\omega$-Lie algebras 
(actually, he fully classified them over the field of real numbers).

It seems that no structures like this were studied before. 
Of course, altered Jacobi identities appeared previously in the literature,
the closest things we are aware of are, first, algebras studied by Sagle in a series of papers started in the 1960s
(see, for example, \cite{sagle} and references therein),
second, structures which, as we suspect, started to appear a long time ago in the 
literature (see, for example, \cite{liu}), and recently were advertised and systematically studied by Hartwig, Larsson and Silvestrov in \cite{hls} under the name of Hom-Lie algebras, and, third, $L_\infty$-algebras and their relatives.
Sagle's algebras are obtained by taking the direct sum decomposition $L = H \oplus M$ 
of a Lie algebra $L$, where $H$ is a subalgebra, $[H,M] \subseteq M$, 
and defining a new algebra structure on $H$ as the projection of the Lie bracket on it.
Such algebras satisfy the condition
$$
[[x,y],z] + [[z,x],y] + [[y,z],x] = [h(x,y),z] + [h(z,x),y] + [h(y,z),x]
$$
where $h: H \times H \to M$ is the projection of the Lie bracket on $M$.
Hom-Lie algebras satisfy the condition
$$
[[x,y],z] + [[z,x],y] + [[y,z],x] = [[x,y],\sigma(z)] + [[z,x],\sigma(y)] + [[y,z],\sigma(x)]
$$
where $\sigma: L \to L$ is a linear map.
In both of these cases, the Jacobi identity is altered by maps to the underlying 
algebra, while the $\omega$-Jacobi identity is altered by the map $\omega$
to the ground field, so their similarity is probably too superficial. 
In a sense, the $\omega$-Jacobi identity should be much more restrictive.

$L_\infty$-algebras are much more general structures encompassing the notion of a 
(co)chain complex and a Lie algebra. The Jacobi identity in these structures is valid 
``up to homotopy'' (see for example, the conditions defining the so-called $2$-term
$L_\infty$-algebras in \cite[Lemma 4.3.3]{baez}, especially  the condition (g)). 
However, a tedious but straightforward computation which we will omit here, shows that this 
``homotopy'', in general, cannot take the form of the right-hand side of (\ref{jacobi}), so 
$\omega$-Lie algebras cannot serve as initial terms of $L_\infty$-algebras, except
for some degenerate trivial cases.

Unlike most of the classes of algebras studied, the $\omega$-Jacobi identity does not
single out a variety of algebras. In fact, the class of $\omega$-Lie algebras
is not closed under the usual constructions employed in structural
theory of algebras, such as taking the direct sum or tensoring with commutative associative algebra.
(It is however closed under taking subalgebras and quotients. The first fact is obvious,
the second one is not and comes after a bit of additional work, as shown below).
Moreover, the $\omega$-Jacobi identity suggests 
that any $\omega$-Lie algebra with nontrivial $\omega$ 
should be close to a perfect one ($L = [L,L]$), thus largely excluding phenomena
related to nilpotency and solvability.
Such algebras cannot be graded with a large number of graded components, 
so an analog of the root space decomposition with respect to a Cartan subalgebra,
if it exists, should have properties different from the Lie-algebraic case.
%

The main result of this paper roughly says: finite-dimensional
$\omega$-Lie algebras which are not Lie algebras are either low-dimensional,
or possess a very ``degenerate'' structure -- in particular, 
have an abelian subalgebra of small 
codimension with further restrictive conditions.

In the first two short sections of this paper we observe some elementary,
but useful facts about ground field extension and modules over $\omega$-Lie algebras, 
needed in subsequent sections.
In \S 3 we establish a sort of analog of the $\omega$-Jacobi identity in 4 variables
(Lemma \ref{4-var}) which will serve as our main working tool, 
and establish with its help some auxiliary facts about ideals of $\omega$-Lie algebras.
\S 4 contains a treatment of a rudimentary analog of the root space decomposition.
\S 5 contains results about quasi-ideals, and
establishes a preliminary division of finite-dimensional $\omega$-Lie algebras 
(Lemma \ref{main-lemma}) into the following three classes: 
those having a Lie subalgebra of codimension $1$, those having $Ker\,\omega$
of codimension $2$, and those of the form
of an abelian extension of a simple $\omega$-Lie algebra with a
non-degenerate $\omega$.

The next three sections contain treatments of these three classes.
Though we are unable to achieve a complete classification
(and doubt a reasonable classification exists), we show that all $\omega$-Lie
algebras under consideration are ``degenerate'' in the sense that they contain
an abelian subalgebra of a small codimension.
In \S 8 we prove that $\omega$-Lie algebras with a nondegenerate $\omega$ do not exist,
thus completing the classification.
\S 9 contains formulations of two main theorems, which describe the structure of 
finite-dimensional $\omega$-Lie algebras, and claim that there are no semisimple 
$\omega$-Lie algebras (which are not Lie
algebras) in high dimensions.
In \S 10 we discuss what identities may be satisfied by $\omega$-Lie algebras,
and the last \S 11 contains some further questions and speculations.
Appendix contains description of the GAP code used in analysis of low-dimensional
algebras in \S 8.

We note that in the course of the study of $\omega$-Lie algebras many notions
in the Lie algebras theory -- derivations, second cohomology, quasi-ideals -- arise naturally.

\section*{Notation and conventions}

The ground field $K$ is assumed to be an arbitrary field 
of characteristic different from 2 and 3, unless stated otherwise. 

Our terminology concerning bilinear forms is standard. Let $\omega$
be a skew-sym\-met\-ric bilinear form on a linear space $V$.
A subspace $W \subseteq V$ is called \textit{isotropic} if $\omega(W,W) = 0$.
Let $W^\bot = \set{x\in V}{\omega(x,W) = 0}$ denote the orthogonal complement to a subspace $W$.
Let $Ker\,\omega = V^\bot$ denote the kernel of $\omega$.
For the standard results from linear algebra we use, see, for example, \cite{bourbaki}.

The Lie-algebraic notions that do not involve the form $\omega$ in
their definitions are extended verbatim to $\omega$-Lie algebras:
for example, we speak about commutators, commutant, adjoint endomorphisms (which are 
right multiplications), subalgebras, ideals, simple, semisimple and abelian algebras, and nilpotent elements.


\section{Extension of the ground field}

Sometimes in the subsequent reasonings, we, naturally, would like to have the luxury to work over
an algebraically closed ground field. For this, one should to be sure that the property
of being an $\omega$-Lie algebra is preserved under the ground field extension. 
This is indeed the case, as the following elementary proposition shows.

\begin{proposition}\label{ext}
Let $L$ be an algebra over a field $K$, $K \subset F$ a field extension.
Then $L$ is an $\omega$-Lie algebra over $K$ for some bilinear form $\omega$ on $L$ if and only if 
$L\otimes_K F$ is an $\Omega$-Lie algebra over $F$ 
for some bilinear form $\Omega$ on $L\otimes_K F$.
\end{proposition}

\begin{proof}
The ``only if'' part is obvious:
$L\otimes_K F$ is an $\Omega$-Lie algebra where $\Omega$ is a bilinear form
on $L\otimes_K F$ extended from $\omega$ by linearity. 

To see the validity of the ``if'' part, note that if $\dim L \le 2$, the statement is 
trivially true
(both $L$ and $L\otimes_K F$ are Lie algebras and both $\omega$ and $\Omega$ can be 
chosen arbitrary),
and assume $\dim L \ge 3$. Take any linearly independent elements $x, y, z\in L$, 
and apply the $\omega$-Jacobi identity to the triple 
$x\otimes 1, y\otimes 1, z\otimes 1 \in L\otimes_F K$.
Then the left-hand side of the $\omega$-Jacobi identity lies in $L\otimes 1$, hence 
all coefficients on the right-hand side
belong to $K$. Hence $\Omega(L\otimes 1, L\otimes 1) \subseteq K$, and we may take $\omega$
to be a restriction of $\Omega$ to $L\otimes 1$.
\end{proof}

\section{Modules}\label{coho}

Let $L$ be an $\omega$-Lie algebra. Consider a vector space $M$ over $K$ and a linear 
homomorphism $\varphi: L \to End(M)$. It is natural to assume that $M$ is an $L$-module,
if the semidirect
product $L \oplus M$, with multiplication extended from $L$ by $[x,m] = \varphi(x)m$, 
$x\in L$, $m\in M$, and $[M,M] = 0$, and a skew-symmetric bilinear form $\Omega$ extended
from $L$, is an $\Omega$-Lie algebra. One immediately sees that, provided
$\dim L \ge 2$, this is the case if and only if
\begin{equation}\label{rep}
\varphi([x,y])m = \varphi(x)\varphi(y)m - \varphi(y)\varphi(x)m + \omega(x,y)m
\end{equation}
for any $x,y\in L$, $m\in M$, and $\Omega$ that trivially extends $\omega$: 
$M \subseteq Ker\,\Omega$.

This suggests the following

\begin{definition}
A vector space $M$ is called a \textit{module over an $\omega$-Lie algebra} $L$, if
there exists a homomorphism $\varphi: L \to End(M)$ such that {\rm (\ref{rep})} holds.
\end{definition}

Note that the very existence of a module over an $\omega$-Lie algebra could impose severe
restrictions on it. For example, consider the case of a $1$-dimensional module $M = Km$. 
Then 
\begin{equation}\label{lambda}
\varphi(x) = \lambda(x)m
\end{equation}
for some linear form $\lambda: L\to K$, any two endomorphisms $\varphi(x)$, where 
$x\in L$, commute, and (\ref{rep}) reduces to
\begin{equation}\label{mult}
\omega(x,y) = \lambda([x,y]).
\end{equation}

This is an important case we will encounter below, so it deserves a special
\begin{definition}
An $\omega$-Lie algebra is called \textit{multiplicative} if there is a linear form $\lambda: L \to K$
such that {\rm (\ref{mult})} holds, i.e.,
$$
[[x,y],z] + [[z,x],y] + [[y,z],x] = \lambda([x,y]) z + \lambda([z,x]) y + \lambda([y,z]) x
$$
for any $x,y,z \in L$.
\end{definition}

So, the previous observation could be rephrased as
\begin{lemma}
An $\omega$-Lie algebra $L$ has a $1$-dimensional module if and only if $L$ is multiplicative,
in which case the module structure is given by {\rm (\ref{lambda})}.
\end{lemma}

Note that, unless $L$ is a Lie algebra, $L$ is not a module over itself under the
adjoint action.

As in the case of Lie algebras, we may consider extensions of an $\omega$-Lie algebra $L$
by means of an $L$-module $M$:
\begin{equation}\label{ext-1}
0 \to M \to E \to L \to 0
\end{equation}
where $M$ is considered as an abelian algebra, and $\omega$ is extended from $L$ to $E$ trivially
by putting $\omega(M,L) = 0$.
In what follows, we will need only the following case which we distinguish by the 
following 
\begin{definition}
An \textit{abelian extension} of an $\omega$-Lie algebra $L$ is an extension of $L$ by the 
direct sum of several copies of a $1$-dimensional $L$-module\footnote[2]{
Note that this definition does not match the case of Lie algebras,
where any extension of type (\ref{ext-1}) is called \textit{abelian}. The closest case in Lie algebras
would be \textit{central extensions}, but the term \textit{central} is obviously inappropriate here
as a $1$-dimensional module is necessarily non-trivial in the non-Lie case. 
I was not imaginative enough to devise a new
term. As we consider in this paper 
only extensions of $\omega$-Lie algebras which are not Lie algebras, this should
not lead to confusion.}. 
\end{definition}

Given an $\omega$-Lie algebra and an $L$-module $M$, one may define cohomology groups
$H^n(L,M)$ precisely by the same formula for the differential as for ordinary Lie algebras. 
Direct, but tedious calculation shows that the square of the differential is zero, 
so this cohomology is well-defined.
As in the case of Lie algebras, direct verification shows that $H^2(L,M)$ describes 
nonequivalent classes of extensions of kind (\ref{ext-1}), and, consequently,
abelian extensions of $L$ are described 
by the direct sum of an appropriate number of copies of $H^2(L,K)$, where $K$ is 
understood as a $1$-dimensional $L$-module.

\section{Ideals}\label{section-ideals}

In this section we show that ideals of non-Lie $\omega$-Lie algebras
are either ``large'', or have a very simple structure.

\begin{lemma}\label{ideals}
Let $I$ be a proper ideal of an $\omega$-Lie algebra $L$.
Then $\omega(I,I) = 0$. If, additionally, $I$ is of codimension $>1$, 
then $I \subseteq Ker\,\omega$.
\end{lemma}

\begin{proof}
Apply the $\omega$-Jacobi identity to $x, y\in I$ and $z \notin I$, $z \ne 0$. 
All the terms on the left-hand side belong to $I$, and the terms 
$\omega(z,x)y$ and $\omega(y,z)x$ on the right-hand side also belong to $I$. Hence the
remaining term $\omega(x,y)z$ belongs to $I$. Hence, $\omega(x,y) = 0$ for any $x,y \in I$.

Now look again at the $\omega$-Jacobi identity with $x \in I$. All the terms on the 
left-hand side still
belong to $I$, as well as the term $\omega(y,z)x$ on the right-hand side. Hence,
$\omega(x,y)z - \omega(x,z)y \in I$ for any $y,z \in L$.
If codimension of $I$ is $>1$, this obviously implies $\omega (I,V) = 0$
for any subspace $V$ of $L$ complementary to $I$. Together with $\omega (I,I) = 0$,
this implies $\omega(I,L) = 0$.
\end{proof}

\begin{corollary}\label{ideal} 
A proper ideal of an $\omega$-Lie algebra is a Lie algebra. 
\end{corollary}

Note that condition $I \subseteq Ker\,\omega$ ensures that one can define an induced form $\omega$
on the quotient space $L/I$, which obviously satisfies the $\omega$-Jacobi identity, so 
a quotient of an $\omega$-Lie algebra by an ideal of codimension $>1$
is an $\omega$-Lie algebra.

Lemma \ref{ideals} suggests to consider the cases of ideals of codimension $1$
and of codimension $>1$ separately. The former case will be considered in 
\S \ref{sec-der}.

We continue with the following Lemma, which, together with the $\omega$-Jacobi identity,
will be our main tool in deriving properties of $\omega$-Lie algebras.

\begin{lemma}\label{4-var}
Let $L$ be an $\omega$-Lie algebra. Then, for any $x,y,z,t \in L$, the following holds:
\begin{multline}\label{two-basic}
  \omega(z,t)[x,y] + \omega(t,y)[x,z] + \omega(y,z)[x,t] + \omega(x,t)[y,z] + \omega(z,x)[y,t] 
+ \omega(x,y)[z,t] \\ =
d\omega(t,z,y)x + d\omega(z,t,x)y + d\omega(y,x,t)z + d\omega(x,y,z)t,
\end{multline}
where $d\omega(x,y,z) = \omega([x,y],z) + \omega([z,x],y) + \omega([y,z],x)$.
\end{lemma}

\begin{proof}
Write the $\omega$-Jacobi identity for triples $x, y, [z,t]$ and $[x,y], z, t$:
\begin{align*}
[[x,y],[z,t]] + [[[z,t],x],y] - [[[z,t],y],x] &= \omega(x,y)[z,t] + \omega([z,t],x)y + \omega(y,[z,t])x \\
[[[x,y],z],t] - [[[x,y],t],z] - [[x,y],[z,t]] &= \omega([x,y],z)t + \omega(t,[x,y])z + \omega(z,t)[x,y]
\end{align*}
and sum the two equalities obtained:
\begin{multline}\label{sum}
[[[z,t],x],y] - [[[z,t],y],x] + [[[x,y],z],t] - [[[x,y],t],z] \\ =
  \omega(x,y)[z,t] + \omega([z,t],x)y + \omega(y,[z,t])x + \omega([x,y],z)t 
+ \omega(t,[x,y])z + \omega(z,t)[x,y] .
\end{multline}
Multiply the $\omega$-Jacobi identity for $x,y,z$ by $t$:
\begin{equation}\label{byt}
[[[x,y],z],t] + [[[z,x],y],t] + [[[y,z],x],t] = 
\omega(x,y)[z,t] + \omega(z,x)[y,t] + \omega(y,z)[x,t].
\end{equation}
Subtract (\ref{byt}) from (\ref{sum}):
\begin{multline*}
[[[z,t],x],y] - [[[z,t],y],x] - [[[x,y],t],z] - [[[z,x],y],t] - [[[y,z],x],t] \\ =
  \omega([z,t],x)y + \omega(y,[z,t])x + \omega([x,y],z)t + \omega(t,[x,y])z \\
+ \omega(z,t)[x,y] - \omega(z,x)[y,t] - \omega(y,z)[x,t].
\end{multline*}
Perform cyclic permutations of $x,y,t$ in the last equality and sum the three equalities
so obtained:
\begin{multline*}
- [[[x,y],t] + [[t,x],y] + [[y,t],x], z] \\ = 
  d\omega(t,z,y)x + d\omega(z,t,x)y + d\omega(y,x,t)z + d\omega(x,y,z)t   \\
- \omega(y,z)[x,t] - \omega(z,t)[x,y] - \omega(x,z)[t,y].
\end{multline*}
Combining the last equality with the $\omega$-Jacobi identity for $x,y,t$,
we get the equality desired.
\end{proof}

An alternative way to derive the identity (\ref{two-basic}), based on superalgebras, is 
outlined in \S \ref{ident}.

We need also the following auxiliary technical

\begin{lemma}\label{aux}
Let $L$ be an $\omega$-Lie algebra and $I$ a nonzero linear subspace of $Ker\,\omega$ such that
$$
[\omega(y,z)x + \omega(z,x)y + \omega(x,y)z, h] \in Kh  
$$
for any $x,y,z\in L$, $h\in I$. Then one of the following holds:
\begin{enumerate}
\item $L$ is multiplicative and $I$ is an abelian ideal of $L$ which, as an $L/I$-module,
      is isomorphic to the direct sum of $1$-dimensional modules.
\item $I$ is contained in a Lie subalgebra of $L$ of codimension $1$.
\item $Ker\,\omega$ is a Lie subalgebra of $L$ of codimension $2$,
      $Ker\,\omega = \set{x\in L}{[x,h] \in Kh}$ for some $h\in I$,
      and $[[Ker\,\omega, Ker\,\omega],h] = 0$.
\item $L$ is a Lie algebra. 
\end{enumerate}
\end{lemma}

\begin{proof}
Denote $N(h) = \set{x \in L}{[x,h] \in Kh}$.
Writing the $\omega$-Jacobi identity for $x,y\in N(h)$ and $h\in I$, we get:
\begin{equation}\label{yoyo-5}
[[x,y],h] = \omega(x,y)h.
\end{equation}
Hence $N(h)$ is a subalgebra of $L$ for any $h\in I$.

We have:
\begin{equation}\label{ins}
\omega(y,z)x + \omega(z,x)y + \omega(x,y)z \in N(h)
\end{equation}
for any $x,y,z\in L$ and $h\in I$.
Letting here $z\in N(h)$, we get
\begin{equation}\label{yoyo-6}
\omega(y,z)x + \omega(z,x)y \in N(h)
\end{equation}
for any $x,y\in L$, and letting further $y \in N(h)$, we get
$$
\omega(y,z)x \in N(h)
$$
for any $x\in L$. The last inclusion implies that either $N(h) = L$, 
or $\omega(N(h),N(h)) = 0$. 

If $N(h) = L$ for all $h \in I$, then $[x,h] = \lambda(x,h)h$
for any $x\in L$, $h\in I$ and some map $\lambda: L \times I \to K$. 
Obviously $I$ is an ideal of $L$. 
By linearity, $\lambda$ is linear in the first argument and constant in the second one,
so we may write $\lambda(x, \cdot) = \lambda(x)$.
By (\ref{yoyo-5}), $\omega(x,y) = \lambda([x,y])$ for any $x,y \in L$, 
so $L$ is multiplicative. As $0 = [h,h] = \lambda(h)h$ for any $h\in I$, we have
$\lambda(I) = 0$ and $I$ is abelian, so we are in case (i). 

Assume now there is $h\in I$ such that $\omega(N(h),N(h)) = 0$, so $N(h)$ is a proper 
Lie subalgebra of $L$. 
By (\ref{yoyo-6}), either $N(h)$ is of codimension $1$, or $\omega(N(h),L) = 0$.
Let $N(h)$ be of codimension $1$. If $I\subseteq N(h)$, we are in (ii).
If $I \nsubseteq N(h)$, then $L = N(h) + I$. But then $\omega(N(h),N(h)) = 0$ and 
$I\subseteq Ker\,\omega$ imply $\omega(L,L) = 0$, hence $L$ is a Lie algebra
and we are in case (iv).

If $\omega(N(h),L) = 0$, (\ref{ins}) implies that either $N(h)$ is of codimension $2$, 
or $\omega(L,L) = 0$, i.e. $L$ is a Lie algebra again.

So the only case remained to consider is when $N(h)$ is of codimension $2$ and lies in $Ker\,\omega$
for some $h\in I$. If $L$ is not a Lie algebra, i.e., $Ker\,\omega$ is proper, 
then $N(h) = Ker\,\omega$. By (\ref{yoyo-5}), $[[N(h),N(h)],h]$ = 0, and we are in case (ii).
\end{proof}

\begin{corollary}\label{ideals-codim-g1}
Let $L$ be an $\omega$-Lie algebra and $I$ a nonzero ideal of $L$ of codimension $>1$. 
Then the conclusion of Lemma {\rm \ref{aux}} holds.
\end{corollary}

\begin{proof}
By Lemma \ref{ideals}, $I\subseteq Ker\,\omega$.
Write (\ref{two-basic}) for $x,y,z\in L$, $h\in I$:
\begin{equation*}
\omega(y,z)[x,h] + \omega(z,x)[y,h] + \omega(x,y)[z,h] = d\omega(x,y,z)h.
\end{equation*}
Thus Lemma \ref{aux} is applicable. 
\end{proof}

\section{Rudimentary root space decomposition}

We start this section with another application of Lemma \ref{4-var}.

\begin{lemma}\label{aux-1}
Let $L$ be an $\omega$-Lie algebra and $H$ an abelian subalgebra of 
$Ker\,\omega$ of dimension $>1$. Then:
\begin{enumerate}
\item $\omega([x,h],y) + \omega(x,[y,h]) = 0$
\item $[\omega(y,z)x + \omega(z,x)y + \omega(x,y)z,h] = d\omega(x,y,z)h$
\end{enumerate}
for any $x,y,z \in L$, $h\in H$.
\end{lemma}

\begin{proof}
Write (\ref{two-basic}) for $x,y\in L, h,h^\prime\in I$:
$$
(\omega([h^\prime,y],x) + \omega([x,h^\prime],y)) h + 
(\omega([h,x],y)        + \omega([y,h],x))        h^\prime = 0.
$$ 
Choosing $h$ and $h^\prime$ to be linearly independent, we arrive at case (i).

Now writing (\ref{two-basic}) for $x,y,z \in L$, $h\in I$, 
and taking into account (i), we arrive at case (ii).
\end{proof}

In particular, (ii) shows that Lemma \ref{aux} is applicable:

\begin{corollary}\label{ker}
Let $L$ be an $\omega$-Lie algebra and $I$ an abelian subalgebra of 
$Ker\,\omega$ of dimension $>1$. Then the conclusion of Lemma \ref{aux} holds.
\end{corollary}

We see that $Ker\,\omega$ in the $\omega$-Lie algebra satisfies, in general, 
quite restrictive conditions. However, to treat the cases (ii) and (iii) of Lemma \ref{aux}
in an uniform way, we continue to consider some generalities about
$Ker\,\omega$.

The following two Lemmata are analogs of the facts used in the proof of the 
well-known properties of root space decompositions in Lie algebras.
Not surprisingly, they feature very similar inductive proofs involving binomial coefficients.

\begin{lemma}
Let $L$ be an $\omega$-Lie algebra, $H$ an abelian Lie subalgebra of $Ker\,\omega$, and
$\dim H > 1$. Then
\begin{equation}\label{binom}
\sum_{i=0}^n \binom{n}{i} \omega \Big( (\ad(h) + \alpha)^{n-i}(x), (\ad(h) + \beta)^i(y) \Big) = 
(\alpha + \beta)^n \omega(x,y)
\end{equation}
for any $n\in \mathbb N$, $x,y\in L$, $h\in H$, $\alpha, \beta\in K$.
\end{lemma}

\begin{proof}
Induction on $n$. The case $n=1$ follows easily from Lemma \ref{aux-1}(i).

Writing (\ref{binom}) for a given $n$
for pairs $(\ad(h) + \alpha)x, y$ and $x, (\ad(h) + \beta)y$ and summing the two  
equalities obtained, we get on the left-hand side:
\begin{multline*}
\sum_{i=0}^n \binom{n}{i} \omega \Big( (\ad(h) + \alpha)^{n+1-i}(x), (\ad(h) + \beta)^i(y) \Big) 
\\ +
\sum_{i=0}^n \binom{n}{i} \omega \Big( (\ad(h) + \alpha)^{n-i}(x), (\ad(h) + \beta)^{i+1}(y) \Big)
\\ =
\sum_{i=0}^{n+1} \Big(\binom{n}{i} + \binom{n}{i-1}\Big) 
\omega \Big( (\ad(h) + \alpha)^{n+1-i}(x), (\ad(h) + \beta)^i(y) \Big) 
\\ =
\sum_{i=0}^{n+1} \binom{n+1}{i} 
\omega \Big( (\ad(h) + \alpha)^{n+1-i}(x), (\ad(h) + \beta)^i(y) \Big) 
\end{multline*}
and on the right-hand side:
$$
(\alpha + \beta)^n \omega((\ad(h) + \alpha)x, y) 
+
(\alpha + \beta)^n \omega(x, ((\ad(h) + \beta))y) 
= 
(\alpha + \beta)^{n+1} \omega(x,y).
$$
This provides the induction step.
\end{proof}

\begin{lemma}
Under the same conditions as in the previous Lemma,
\begin{multline}\label{binom-mult}
\sum_{i=0}^n \binom{n}{i} \Big[(\ad(h) + \alpha)^{n-i}(x), (\ad(h) + \beta)^i(y)\Big] \\ = 
(\ad(h) + \alpha + \beta)^n ([x,y]) - n(\alpha + \beta)^{n-1} \omega(x,y)h
\end{multline}
for any $n\in \mathbb N$, $x,y\in L$, $h\in H$, $\alpha, \beta\in K$.
\end{lemma}

\begin{proof}
Induction on $n$. The case $n=1$ is verified directly using the $\omega$-Jacobi identity
for $x,y,h$ and Lemma \ref{aux-1}(i). The induction step runs as follows.

Writing (\ref{binom-mult}) for a given $n$ for pairs
$(\ad(h) + \alpha)x, y$ and $x, (\ad(h) + \beta)y$, and summing the two 
equalities obtained, we get on the left-hand side:
\begin{multline*}
\sum_{i=0}^n \binom{n}{i} \Big[ (\ad(h) + \alpha)^{n+1-i}(x), (\ad(h) + \beta)^i(y) \Big]
\\ +
\sum_{i=0}^n \binom{n}{i} \Big[ (\ad(h) + \alpha)^{n-i}(x), (\ad(h) + \beta)^{i+1}(y) \Big]
\\ =
\sum_{i=0}^{n+1} \Big(\binom{n}{i} + \binom{n}{i-1}\Big) 
\Big[ (\ad(h) + \alpha)^{n+1-i}(x), (\ad(h) + \beta)^i(y) \Big]
\\ =
\sum_{i=0}^{n+1} \binom{n+1}{i} 
\Big[ (\ad(h) + \alpha)^{n+1-i}(x), (\ad(h) + \beta)^i(y) \Big]
\end{multline*}
and on the right-hand side:
\begin{multline*}
((\ad(h) + \alpha + \beta)^n ([(\ad(h) + \alpha)x, y] + [x,(\ad(h) + \beta)y]) \\ -
n(\alpha + \beta)^{n-1} (\omega((\ad(h) + \alpha)x,y) + \omega(x,(\ad(h) + \beta)y))h \\ =
((\ad(h) + \alpha + \beta)^{n+1}([x,y]) - \omega(x,y) ((\ad(h) + \alpha + \beta)^n h -
n(\alpha + \beta)^n \omega(x,y)h \\ =
((\ad(h) + \alpha + \beta)^{n+1}([x,y]) - (n+1)(\alpha + \beta)^n \omega(x,y)h.
\end{multline*}
\end{proof}

Assume the ground field $K$ is algebraically closed. As $\ad(H)$ is a space of commuting 
endomorphisms of $L$, we may consider
the root space decomposition of $L$ with respect to $\ad(H)$:
\begin{equation}\label{root}
L = L_0 \oplus \bigoplus_\alpha L_\alpha.
\end{equation}
As in the Lie algebras case, we will write $L_\alpha$ for any $\alpha\in H^*$,
assuming it being zero if $\alpha$ is not a root.
Let $L^\alpha = \set{x\in L}{[x,h] = \alpha(h)x \text{ for all } h\in H}$ 
denotes the simple subspace of the root space $L_\alpha$.

\begin{lemma}\label{roots}
Let $L$ be a finite-dimensional $\omega$-Lie algebra over an algebraically closed field,
$H$ an abelian subalgebra of $Ker\,\omega$, $\dim H > 1$,
and {\rm (\ref{root})} is the root space decomposition
of $L$ with respect to $H$. Then:
\begin{enumerate}
\item $\omega(L_\alpha, L_\beta) = 0$ for any two roots $\alpha$, $\beta$ such that 
      $\alpha + \beta \ne 0$.
\item $[L_\alpha, L_\beta] \subseteq L_{\alpha + \beta}$ for any two roots $\alpha$, $\beta$.
\item If for some nonzero root $\alpha$, there is a root $-\alpha$, then either
      there are no more nonzero roots, or both $L_\alpha$ and $L_{-\alpha}$ lie in $Ker\,\omega$.
\item If $L_0 = H$, then $H \oplus \bigoplus_\alpha L^{\alpha}$ is a Lie subalgebra of $L$.
\end{enumerate}
\end{lemma}

\begin{proof}
(i)
Take $x\in L_\alpha$ and $y\in L_\beta$. Then (\ref{binom}) implies that for a sufficiently 
large $n$ and any $h\in H$, 
$$
(-\alpha(h) - \beta(h))^n \omega(x,y) = 0 .
$$
Hence $\omega (L_\alpha, L_\beta) = 0$ if $\alpha + \beta \ne 0$.

(ii)
In its turn, (\ref{binom-mult}) shows that for a sufficiently large $n$, 
$$
(\ad(h) - (\alpha(h) + \beta(h)))^n ([x,y]) = n(-\alpha(h) - \beta(h))^{n-1} \omega(x,y)h.
$$
The right-hand side here vanishes for any $\alpha$, $\beta$, as by just proved if 
$\alpha + \beta \ne 0$, then $\omega(x,y) = 0$.
Hence $[L_\alpha,L_\beta] \subseteq L_{\alpha + \beta}$.

(iii)
Suppose there are three distinct nonzero roots $\alpha, -\alpha, \beta$, and take $x\in L_\alpha$, 
$y\in L_{-\alpha}$, $z\in L_\beta$. Applying Lemma \ref{aux-1}(ii),
we see that all summands in the corresponding equality, 
lying in different root spaces, vanish. In particular, $[z,h]\omega(x,y) = 0$.
Choosing $z$ to be an eigenvector from the corresponding root space, i.e., $[z,h] = \beta(h)z$,
we see that $\omega(L_\alpha, L_{-\alpha}) = 0$. Together with (i) this implies
that both $L_\alpha$ and $L_{-\alpha}$ lie in $Ker\,\omega$.

(iv)
It is clear that $H \oplus \bigoplus_\alpha L^{\alpha}$ is a subalgebra.
Writing the $\omega$-Jacobi identity for $x\in L^\alpha$, $y\in L^{-\alpha}$, $h\in H$,
we get $\omega(x,y) = 0$. This shows that $\omega(L^\alpha, L^{-\alpha}) = 0$, which
together
with (i) implies that $\omega$ vanishes on $H \oplus \bigoplus_\alpha L^{\alpha}$.
\end{proof}

It is possible to ponder this situation further to get more exotic-looking 
properties of root systems in non-Lie $\omega$-Lie algebras, but no need in that: after all,
this machinery will be applied below only to quite degenerate situations 
when codimension of $H$ is small.

\section{Kernel and quasi-ideals}

The aim of this section is to show that in nontrivial cases, the form $\omega$
should satisfy very strong vanishing conditions, and establish a preliminary
classification of $\omega$-Lie algebras.

\begin{lemma}\label{xy}
Let $L$ be a finite-dimensional $\omega$-Lie algebra, and $x,y\in L$. Then $[x,y] \in Kx + Ky$
in each of the following cases:
\begin{enumerate}
\item $x,y \in Ker\,\omega$ and $rank(\omega) \ge 2$.
\item $x \in Ker\,\omega$ and $rank(\omega) \ge 4$.
\item $\omega(x,y)=0$ and $rank(\omega) \ge 6$.
\end{enumerate}
\end{lemma}

\begin{proof}
All the cases follow the same format with slight modifications. We use (\ref{two-basic})
for suitably chosen $z$ and $t$. The condition of vanishing of $\omega$ ensures that all but one 
terms on the left-hand side vanish, and, applying $\omega(\cdot,z)$ to both sides,
we derive further vanishing of the corresponding terms on the right-hand side.

(i) Choose $z,t\in L$ such that $\omega(z,t) = 1$. In that case (\ref{two-basic}) gives:
\begin{equation}\label{yoyo-0}
[x,y] = d\omega(t,z,y)x + d\omega(z,t,x)y + \omega([y,x],t)z + \omega([x,y],z)t.
\end{equation}
Applying to both sides of this equality $\omega(\cdot, z)$, we get
$\omega([x,y],z) = -\omega([x,y],z)$, whence $\omega([x,y],z) = 0$. Similarly, 
$\omega([y,x],t) = 0$ and (\ref{yoyo-0}) reduces to the desired condition.

(ii) We may assume $y\notin Ker\,\omega$, otherwise we are covered by case (i). 
Let $y\in V$ for a certain linear complement $V$ of $Ker\,\omega$ in $L$, 
$\omega$ being nondegenerate on $V$. Since $Ky$
is an $1$-dimensional isotropic subspace of $V$, it lies in a certain maximal isotropic
subspace $W$. Then there is a symmetric nondegenerate bilinear form 
on $W$ such that $V = W \oplus W^*$, $W^*$ is a conjugate of $W$ with respect to 
that form, and 
$$
\omega(a + f, a^\prime + f^\prime) = f^\prime(a) - f(a^\prime)
$$
for $a,a^\prime \in W, f,f^\prime \in W^*$.

As $\dim W = \frac{1}{2} rank(\omega) \ge 2$, we may take $z\in W$ linearly independent with 
$y$. Take $t=z^*$. In that case (\ref{two-basic}) gives:
\begin{multline}\label{yoyo-4}
[x,y] = d\omega(t,z,y)x + d\omega(z,t,x)y \\ +  
(\omega([y,x],t) + \omega([x,t],y))z + 
(\omega([x,y],z) + \omega([z,x],y))t.
\end{multline}
Applying to both sides of this equality $\omega(\cdot,z)$, we get:
$$
\omega([x,y],z) = -\omega([x,y],z) - \omega([z,x],y) ,
$$
whence
$$
2\omega([x,y],z) - \omega([x,z],y) = 0 .
$$

By symmetry considerations, interchanging $y$ and $z$, we get
$$
2\omega([x,z],y) - \hfill\linebreak \omega([x,y],z) = 0 ,
$$
whence 
$$
\omega([x,y],z) = \omega([x,z],y) = 0 ,
$$
and the condition desired follows again from (\ref{yoyo-4}).

(iii)
We may assume $x,y\notin Ker\,\omega$, otherwise we are covered by cases (i) and (ii).
We reason as in the previous case, enlarging the isotropic subspace $Kx + Ky$
to a maximal isotropic subspace $W$ in a linear complement $V$ of $Ker\,\omega$.
Since $\dim W = \frac{1}{2} rank(\omega) \ge 3$, we may take $z\in W$ linearly independent 
with $x,y$. Take $t=z^*$. Then (\ref{two-basic}) gives:
$$
[x,y] =
d\omega(t,z,y)x + d\omega(z,t,x)y + d\omega(y,x,t)z + d\omega(x,y,z)t.
$$
Applying $\omega(\cdot,z)$ to both sides of this equality, we get
$$
\omega([x,y],z) = -d\omega([x,y],z) .
$$
Permuting $x,y,z$, we get 
$$
\omega([x,y],z) = d\omega([x,y],z) = 0 ,
$$
and the condition desired readily follows.
\end{proof}

The just proved Lemma shows, in particular, that for a sufficiently large $rank(\omega)$, 
$Ker\,\omega$ is a quasi-ideal of an $\omega$-Lie algebra
(recall that a subspace $I$ of an algebra $L$ is called \textit{quasi-ideal} if
$[I,A] \subseteq I + A$ for any subspace $A\subseteq L$). Quasi-ideals of Lie algebras
were studied by Amayo in \cite[Part I]{amayo}. It is possible to develop a parallel theory of 
quasi-ideals of $\omega$-Lie algebras, but it turns out that it will largely coincide
with the Lie algebras case (which follows, a posteriori, also from the structural results about 
$\omega$-Lie algebras obtained below). Thus we restrict ourselves to the immediate case we need,
namely, of $1$-dimensional quasi-ideals.

\begin{lemma}\label{quasi}
Let $L$ be an $\omega$-Lie algebra, $I$ an $1$-dimensional quasi-ideal of $L$,
$I \subseteq Ker\,\omega$. Then either $I$ is an ideal of $L$, or $L$ is a Lie algebra.
\end{lemma}

\begin{proof}
We chiefly follow the line of reasoning in \cite[pp. 31--32]{amayo}.

If $\dim L = 2$, the Lemma is trivially true, so assume $\dim L \ge 3$.
Let $I = Ka$, $a\in Ker\,\omega$. Then
\begin{equation}\label{yoyo-1}
[x,a] = \lambda(x)x + \mu(x)a
\end{equation}
for some functions $\lambda, \mu: L\to K$ and for any $x\in L$. 
By linearity, $\lambda(x) = \lambda$ is constant.
If $\lambda = 0$, then $Ka$ is an ideal of $L$, so assume $\lambda \ne 0$. Replacing
$a$ by $\frac{1}{\lambda} a$, we may set $\lambda = 1$.
Writing the $\omega$-Jacobi identity for $x,y,a$, and taking into account (\ref{yoyo-1}), we get:
$$
[x,y] = \mu(x) y - \mu(y) x + (\omega(x,y) - \mu([x,y]))a.
$$
Multiplying the last equality by $a$, we get:
$$
[x,y] = \mu(x) y - \mu(y) x - \mu([x,y])a.
$$
Comparing the last two equalities, we get $\omega(x,y) = 0$ for any $x,y\in L$ linearly
independent with $a$. Hence $\omega$ vanishes and $L$ is a Lie algebra.
\end{proof}

Lemma \ref{xy} shows also that any isotropic subspace of an $\omega$-Lie algebra
is a Lie subalgebra in which every $1$-dimensional subspace is a quasi-ideal, 
provided the rank of $\omega$ is large enough. Such Lie algebras have a fairly trivial structure.

\begin{definition}
A semidirect sum $A \oplus Kx$ where $A$ is an abelian Lie algebra and 
$\ad x$ acts on $A$ as the identity map, is called an \textit{almost abelian} Lie algebra,
and $A$ is called its \textit{abelian part}.
\end{definition}

\begin{lemma}\label{2-quasi}
A finite-dimensional Lie algebra such that every two its linearly independent
elements generate a two-dimensional subalgebra, is either abelian or almost abelian.
\end{lemma}

\begin{proof} 
This is implicit in \cite[Part I, Theorem 3.6 and proof of Theorem 3.8]{amayo}.
As the proof is very simple and we will need a similar reasoning later, we will reproduce it 
here.

Let $L$ be a Lie algebra with the property specified in the condition of the Lemma.
We may assume $\dim L > 1$.
Write 
$$
[x,y] = \lambda(x,y)x + \mu(x,y)y
$$
for any two elements $x,y \in L$. By 
anti-commutativity, $\mu(x,y) = -\lambda(y,x)$, and by linearity $\lambda$ is constant
in the first argument, so $[x,y] = \lambda(y)x - \lambda(x)y$
for some linear form $\lambda: L\to K$. If $\lambda = 0$, then $L$ is abelian. 
If $\lambda \ne 0$, write $L = Ker\,\lambda \oplus Kx$ for $x\in L$ such that $\lambda(x) = 1$,
and then $L$ is almost abelian.
\end{proof}

Putting all this together, we get

\begin{lemma}\label{main-lemma}
Let $L$ be a finite-dimensional $\omega$-Lie algebra which is not a Lie algebra. 
Then  one of the following holds:
\begin{enumerate}
\item $L$ has a Lie subalgebra of codimension $1$.
\item 
$Ker\,\omega$ is an abelian or almost abelian Lie subalgebra of $L$ of codimension $2$.
\item 
$L$ is an abelian extension of a simple $\omega$-Lie algebra with nondegenerate $\omega$.
\end{enumerate}
\end{lemma}

\begin{proof}
By Lemmata \ref{xy}(i) and \ref{2-quasi}, $Ker\,\omega$ is an abelian or almost
abelian Lie algebra. If $rank(\omega) = \codim Ker\omega = 2$, we are in case (ii), so assume
$rank(\omega) \ge 4$. Then by Lemma \ref{xy}(ii), $Ker\,\omega$ is an ideal.

If $L$ is simple, then $Ker\,\omega = 0$, which is covered by case (iii). So suppose $L$ is not simple
and consider a nonzero maximal ideal $I$ of $L$. By Corollary \ref{ideal}, $I$ is a Lie algebra.
If $\codim I = 1$, we are in case (i), so let $\codim I > 1$. 
Then by Corollary \ref{ideals-codim-g1} either $L$ is an abelian
extension of a simple $\omega$-Lie algebra $L/I$, or 
$I$ is contained in a Lie subalgebra of codimension $1$.
In the former case, as $rank(\omega|_{L/I}) = rank(\omega) \ge 4$,
by already noted, $\omega$ is nondegenerate on $L/I$, so we are in case (iii).
In the latter case we are in case (i).
\end{proof}

We will treat the cases of Lemma \ref{main-lemma} subsequently in the next three sections.

\section{$(\alpha, \lambda)$-derivations}\label{sec-der}

In the previous sections we had encountered repeatedly a situation when an $\omega$-Lie algebra
has a Lie subalgebra of codimension $1$. In this section we study this situation.
(As, by Corollary \ref{ideal}, proper ideals are necessarily Lie subalgebras, 
this includes also the case of ideals of codimension $1$).

Let $L$ be an $\omega$-Lie algebra and $A$ a subalgebra of $L$ of codimension $1$. Write 
$L = A \oplus Kv$ for some $v \in L$. Then 
\begin{equation}\label{mult-2}
[x,v] = D(x) + \lambda(x)v
\end{equation}
for $x \in A$, and some linear maps $D: A \to A$ and $\lambda: A \to K$. 
Straightforward calculation shows that the $\omega$-Jacobi identity for $L$
is equivalent to the following three conditions: first, $A$ is an $\omega$-Lie algebra, second,
\begin{equation}\label{pre-der}
D([x,y]) - [D(x),y] + [D(y),x] = \lambda(y)D(x) - \lambda(x)D(y) + \omega(y,v)x - \omega(x,v)y,
\end{equation}
and third,
\begin{equation}\label{mult-1}
\omega(x,y) = \lambda([x,y])
\end{equation}
for any $x,y \in A$.

In particular, we have
\begin{lemma}\label{lemma-codim1-mult}
A subalgebra of codimension $1$ in an $\omega$-Lie algebra is a multiplicative $\omega$-Lie 
algebra.
\end{lemma}

Equation (\ref{pre-der}) suggests the following
\begin{definition}
A linear map $D: A \to A$ of an anticommutative algebra $A$ is called 
\textit{$(\alpha, \lambda)$-derivation} of $A$ if there are linear forms $\alpha, \lambda: A \to K$ 
such that 
\begin{equation}\label{der}
D(ab) = D(a)b + aD(b) + \lambda(b)D(a) - \lambda(a)D(b) + \alpha(b)a - \alpha(a)b
\end{equation}
holds for any $a, b \in A$.
\end{definition}

So, given a multiplicative $\omega$-Lie algebra $A$ (with $\omega$ given by (\ref{mult-1}))
and its $(\alpha, \lambda)$-derivation $D$, 
we get an $\omega$-Lie algebra as a vector space $A \oplus Kv$, with multiplication and $\omega$
extended from $A$, and defining the rest by (\ref{mult-2}) and $\omega(x,v) = \alpha(x)$.
Conversely, every $\omega$-Lie algebra with a subalgebra of codimension $1$ occurs in 
that way.
An $\omega$-Lie algebra with a subalgebra $A$ of codimension $1$ is a Lie algebra
if and only if $\alpha$ = 0 and $\lambda([A,A]) = 0$.

Unfortunately, the space of all $(\alpha, \lambda)$-derivations of a given noncommutative algebra
$A$ for a fixed $\lambda$ is, generally, not closed under operation of commutation. 
There is, however, a remarkable case where it does.

\begin{proposition}
The space of all $(\alpha, 0)$-derivations of an anticommutative algebra
forms a Lie algebra under operation of commutation.
\end{proposition}

\begin{proof}
Direct calculation shows that if $D_1$ is an $(\alpha_1, 0)$-derivation and 
$D_2$ is an $(\alpha_2, 0)$-derivation,
then $[D_1, D_2]$ is an $(\alpha_1 \circ D_2 - \alpha_2 \circ D_1, 0)$-derivation.
\end{proof}

This Lie algebra contains an algebra of (ordinary) derivations of $A$.
$(\alpha, 0)$-derivations correspond to the case where $A$ is an ideal
of codimension $1$.

Note that our definition of $(\alpha, \lambda)$-derivations looks somewhat similar 
to some other definitions of generalized derivations of Lie and associative algebras:
generalized derivations in the sense of Leger and Luks, i.e.,
triples $(D_1, D_2, D_3)$ of endomorphisms of an algebra $A$ 
such that $D_1(ab) = D_2(a)b + aD_3(b)$ for any $a,b \in A$
(see \cite{leger-luks})
and generalized derivations in the sense of Nakajima, i.e., pairs $(D, u)$
of an endomorphism $D$ of an algebra $A$ and an element $u\in A$ such that 
$D(ab) = D(a)b + bD(a) + aub$ for any $a,b\in A$
(see, for example, \cite{komatsu-nakajima} and \cite{argac-albas}).
However, this does not go much beyond superficial similarity in formulae:
in general, $(\alpha, \lambda)$-derivations seem to intersect trivially with 
generalized derivations in either sense.

We are interested in $(\alpha, \lambda)$-derivations of Lie algebras. 
In that case, due to (\ref{mult-1}), $\lambda$ vanishes on the commutant of an algebra.

\begin{lemma}\label{der-lie-algs}
Let $L$ be a finite-dimensional Lie algebra and $D$ its $(\alpha, \lambda)$-derivation.
Then one of the following holds:
\begin{enumerate}
\item $\dim L \le 3$.
\item $\alpha = 0$.
\item $Ker\,\alpha$ is a subalgebra of $L$ of codimension $1$ and one of the following holds:
      \begin{enumerate}
      \item $L = A \oplus Kx$, $A$ is abelian, $\ad x: A \to A$ is any linear map; $Ker\,\alpha = A$.
      \item $L$ is the direct sum of an abelian Lie algebra $A$ and the two-di\-men\-si\-o\-nal nonabelian 
            Lie algebra $\langle x, y \,|\, [x,y] = y \rangle$; $Ker\,\alpha = A \oplus Kx$.
      \item $L = A \oplus Kx \oplus Ky$, $A$ is abelian, $\ad x: A \to A$ is the identity map,
            $\ad y: A \to A$ is any linear map, and $[x,y]\in A$; $Ker\,\alpha = A \oplus Kx$.
      \item $L = A \oplus Kx \oplus Ky$, $A$ is abelian, $\ad x: A \to A$ is the identity map, 
            $\ad y: A \to A$ is the zero map, $[x,y] = a + \sigma y$ for some $a\in A$, 
            $\sigma \in K$, $\sigma \ne 0$; $Ker\,\alpha = A \oplus Kx$.
      \end{enumerate}
\end{enumerate}
\end{lemma}

\begin{proof}
Applying $D$ to the Jacobi identity, we get:
\begin{align}
   \alpha(z)[x,y] + \alpha(x)[y,z] &+ \alpha(y)[z,x]                \notag        \\ 
&+ (\alpha([y,z]) + \lambda(y)\alpha(z) - \lambda(z)\alpha(y))x     \label{basic} \\
&+ (\alpha([z,x]) + \lambda(z)\alpha(x) - \lambda(x)\alpha(z))y     \notag        \\
&+ (\alpha([x,y]) + \lambda(x)\alpha(y) - \lambda(y)\alpha(x))z = 0 \notag
\end{align}
for any $x,y,z\in L$.

Assuming in (\ref{basic}) $x,y,z\in Ker\,\alpha$, we get that
either $\dim Ker\,\alpha \le 2$ and hence $\dim L \le 3$, or 
$\alpha([Ker\,\alpha, Ker\,\alpha]) = 0$ and hence $Ker\,\alpha$ is a subalgebra of $L$. 
Thus, assuming $\dim L > 3$, either $\alpha = 0$ or $Ker\,\alpha$ is a subalgebra of $L$ 
of codimension $1$.

Now, taking in (\ref{basic}) $x,y \in Ker\,\alpha$, $z\in L$ such that $\alpha(z) = 1$, we get:
$$
[x,y] = (\alpha([z,y]) - \lambda(y))x + (\alpha([x,z]) + \lambda(x))y.
$$
Thus $Ker\,\alpha$ is a Lie algebra such that any two linearly independent elements in it
generate a two-dimensional subalgebra. By Lemma \ref{2-quasi}, 
$Ker\,\alpha$ is either abelian of almost abelian.
Now straightforward computations produce the list (iii) of Lie algebras having a subalgebra
of codimension $1$ which is either abelian or almost abelian
(note that not all algebras in this list are pairwise non-isomorphic; 
we have accounted also for different possibilities of $Ker\,\alpha$).
\end{proof}

Lie algebras listed in part (iii) may have many 
$(\alpha, \lambda)$-derivations, and they do not appear to allow description in a nice compact form.
For example, consider the algebra in (iiia) with non-nilpotent $\ad x$ and suppose
the ground field is algebraically closed. Let $F$
be an eigenvector corresponding to a nonzero eigenvalue $\sigma$ of $\ad x$ in a Lie algebra 
$End(A)$: $[F, \ad x] = \sigma F$. Then $D \in End(L)$ defined by $D(a) = F(a) + a, a\in A$
and $D(x) = 0$, is an $(\alpha, \lambda)$-derivation for $\alpha$ defined by 
$\alpha(A) = 0$, $\alpha(x) = \sigma$ and $\lambda$ defined by $\lambda(A) = 0, \lambda(x) = -\sigma$.
This provides an example of an $\omega$-Lie algebra which is not a Lie algebra 
in any finite dimension $\ge 3$.

Nevertheless, writing, for each of these cases, the generic 
$(\alpha,\lambda)$-derivations
in terms of linear functions on the corresponding space $A$, the elementary but tedious
linear-algebraic considerations, much in the spirit of 
Lemmata \ref{quasi}, \ref{2-quasi} above, and Lemma \ref{simple-codim2} below,
show that one always can choose a nonzero abelian ideal 
in the corresponding $\omega$-Lie algebra. That leads to

\begin{corollary}\label{simple-codim-1}
A finite-dimensional semisimple $\omega$-Lie algebra with a Lie subalgebra of codimension $1$ is either a Lie
algebra, or has dimension $\le 4$.
\end{corollary}

An alternative proof of this statement might be obtained following the approach which 
was used by Amayo in a description of simple Lie algebras with a subalgebra of codimension $1$ in 
\cite[Part II]{amayo}, but goes back to Weisfeiler at the end of 1960s.
Let $L$ be an $\omega$-Lie algebra of dimension $>3$ with a Lie subalgebra $L_0$ of 
codimension $1$. Since $\dim L_0 \ge 3$, $\omega (L_0,L_0) = 0$. Put $L_{-1} = L$ and
define a filtration on $L$ by 
\begin{equation}\label{filtr}
L_{i+1} = \set{x\in L_i}{[x,L] \subseteq L_i}, \quad i \ge 0 .
\end{equation}
The same reasonings as in the Lie-algebraic case, allow one to establish some of the 
properties of this filtration (for example, that each nonzero term has
codimension $1$ in the preceding term), and the additional condition 
$[[L_0,L_0],[L_0,L_0]] = 0$, which holds by inspection of all Lie algebras listed in 
Lemma \ref{der-lie-algs}(iii), could be used to finish the proof.

Now let us consider some examples of $(\alpha, 0)$-derivations of low-dimensional Lie algebras.
It is obvious that for the $2$-dimensional abelian Lie algebra, all $(\alpha, 0)$-derivations 
are ordinary derivations.
Direct easy calculation shows that if $L$ is the $2$-dimensional nonabelian Lie algebra,
then the Lie algebra of its $(\alpha, 0)$-derivations coincides with $End(L)$. If $x, y$ are 
basic elements of $L$ such that $[x,y] = x$, then the linear transformation given by matrix
$$\begin{pmatrix}
a & b \\
c & d
\end{pmatrix}$$
in that basis is not a derivation if and only if $bd \ne 0$ (and then $\alpha(x) = -b,
\alpha(y) = -d$).

Direct calculation shows that the Lie algebra of $(\alpha, 0)$-derivations of $sl(2)$
is $5$-di\-men\-si\-o\-nal and isomorphic to the semidirect sum of $sl(2)$ and its $2$-dimensional
standard module. In the basis 
$$
\set{e,f,h}{[e,h] = -e, [f,h] = f, [e,f] = h} ,
$$
the basic 
$(\alpha, 0)$-derivations which are not ordinary derivations can be chosen
as 
$$\begin{aligned}
e &\mapsto 0,& f &\mapsto 0,& h &\mapsto e&  \\
e &\mapsto 0,& f &\mapsto 0,& h &\mapsto f ,&
\end{aligned}$$
the corresponding $\alpha$'s being
$$\begin{aligned}
e &\mapsto 0,& f &\mapsto -1,& h &\mapsto 0&   \\
e &\mapsto 1,& f &\mapsto 0,&  h &\mapsto 0 ,&
\end{aligned}$$
respectively.
This provides examples of $4$-dimensional $\omega$-Lie algebras which are not Lie 
algebras. Note that since $\lambda = 0$, these algebras are not simple.

On the other hand, any nonzero $(\alpha,\lambda)$-derivation with $\lambda \ne 0$
of a $3$-dimensional simple $\omega$-Lie non-Lie algebra $A$, gives rise to a simple 
$4$-dimensional $\omega$-Lie algebra. Indeed, if $L$ = $A \oplus Kv$ is such an algebra
with multiplication (\ref{mult-2}), and $I$ is a proper ideal of $L$, then, due to simplicity 
of $A$, $I \cap A = 0$, hence we may write $I = Kv$, which implies $D=0$, a contradiction.

There are many such derivations of $3$-dimensional simple
$\omega$-Lie algebras. This can be verified with the aid of
a simple computer program described in Appendix. 
One such example is given in \S \ref{sect-main}.

\section{Kernel of codimension 2}

Now we can see that in Lemma \ref{main-lemma}, the case (ii) essentially covers,
up to algebras of small dimension, the case (i):

\begin{lemma}
Let $L$ be a finite-dimensional $\omega$-Lie algebra with a Lie subalgebra of 
codimension $1$. Then one of the following holds:
\begin{enumerate}
\item $L$ is a Lie algebra.
\item $\dim L = 3$.
\item $\codim Ker\,\omega = 2$.
\end{enumerate}
\end{lemma}

\begin{proof}
This follows immediately from the results of the previous section. Indeed,
such $\omega$-Lie algebras are described by $(\alpha, \lambda)$-derivations of Lie algebras
listed in Lemma \ref{der-lie-algs}, with $Ker\,\omega = Ker\,\alpha$.
\end{proof}

In the opposite direction we have:

\begin{lemma}\label{codim-2}
Let $L$ be a finite-dimensional $\omega$-Lie algebra with $Ker\,\omega$ of codimension $2$.
Then one of the following holds:
\begin{enumerate}
\item $\dim L = 3$.
\item $L$ has a Lie subalgebra of codimension $1$.
\item $Ker\,\omega$ is almost abelian, with the abelian part acting nilpotently on $L$.
\end{enumerate}
\end{lemma}

\begin{proof}
By \S 5, $Ker\,\omega$ is abelian or almost abelian. 
In the latter case, write $Ker\,\omega = H \oplus Ka$, $H$ is abelian, and 
$[h,a] = h$ for any $h\in H$.
For notational convenience, we will assume $H = Ker\,\omega$ in the case of abelian $Ker\,\omega$.

If $\dim H = 1$, we are in case (i) or (ii), so let $\dim H > 1$. 
Consider the Fitting decomposition of $L$ with respect to $H$: $L = L_0 \oplus L_1$,
and its refinement -- the root space decomposition of 
$\overline{L} = L\otimes_K \overline{K}$ over an algebraic closure $\overline{K}$
of the ground field $K$ with respect to $H\otimes_K \overline{K}$ (note that
$\overline{L}_0 = L_0\otimes_K \overline{K}$). Obviously, $Ker\,\omega \subseteq L_0$.

By Lemma \ref{roots}(ii), $\overline{L}_0$ is a subalgebra of $\overline{L}$, hence
$L_0$ is a subalgebra of $L$. 
Assume first that $L_0 \subsetneqq L$. 
If $Ker\,\omega \subsetneqq L_0$, then $L_0$ is a Lie subalgebra of $L$
of codimension $1$, and we are in case (ii). Hence we may assume $L_0 = Ker\,\omega$
and $\overline{L}_0 = Ker\,\overline{\omega} = Ker\,\omega\otimes_K \overline{K}$.

There is either one nonzero root space $\overline{L}_\alpha$ of dimension $2$, 
or two nonzero root spaces of dimension $1$. 
In the former case, by Lemma \ref{roots}(i),
$\omega(\overline{L}_\alpha, \overline{L}_\alpha) = 0$, hence 
$\overline{L}_\alpha \subseteq Ker\,\overline{\omega}$, a contradiction.
In the latter case, both root spaces are simple. If $Ker\,\omega = H$ is abelian, 
then by Lemma \ref{roots}(iv),
$\overline{L}$ is a Lie algebra, hence $L$ is a Lie algebra, a contradiction.
Suppose $Ker\,\omega$ is almost abelian and let $\overline{L}_{\alpha} = \overline{K}x$ 
be one of the root spaces. 
By Lemma \ref{roots}(ii), we may write $[x,a] = \lambda x$ for some $\lambda\in \overline{K}$. 
Writing the $\omega$-Jacobi identity for $x,a$ and $h\in H$, we get $\alpha(h) = 0$, 
a contradiction.

Now consider the case where $L_0 = L$, i.e. $H$ acts on $L$ nilpotently. 
Suppose $Ker\,\omega = H$ is abelian. $Ker\,\omega$ also acts nilpotently on any module
which is a quotient of $L$, in particular, on 
$L/Ker\,\omega$. Consequently, there is $x\notin Ker\,\omega$ such that the whole
$Ker\,\omega$ maps $x + Ker\,\omega \in L/Ker\,\omega$ to zero, and hence 
$[x,Ker\,\omega] \subseteq Ker\,\omega$. Then $Ker\,\omega \oplus Kx$
is a Lie subalgebra of $L$ of codimension $1$, and we are again in case (ii). 

The remaining case is when $Ker\,\omega$ is almost abelian and $H$ acts on 
$L$ nilpotently, and this is exactly the case (iii).
\end{proof}

Note that the case (iii) does not seem to be amenable to any compact classification.
However, like in \S \ref{sec-der}, we able to deal with simple algebras:

\begin{lemma}\label{simple-codim2}
A finite-dimensional semisimple $\omega$-Lie algebra satisfying the condition (iii) of Lemma 
\ref{codim-2}, has dimension $\le 4$.
\end{lemma}

\begin{proof}
Note that in all $\omega$-Jacobi identities considered below, the right-hand side vanishes,
so, essentially, this proof consists of tedious but elementary Lie-algebraic 
considerations.

Let $L$ be such an algebra. 
Write, as previously, $Ker\,\omega = H \oplus Ka$, $H$ abelian, 
$[a,h] = h$ for any $h\in H$. The set $\ad h$ for all $h\in H$, being a set of 
nilpotent commuting operators on the $2$-dimensional vector space $L/Ker\,\omega$, can be brought
simultaneously to the upper triangular form, i.e., one can choose a basis $\{x,y\}$ of the vector 
space complementary to $Ker\,\omega$ in $L$, and linear functions
$\lambda, \mu, \eta: H \to K$ and $f,g: H\to H$ such that
\begin{align*}
[x,h] &= \lambda(h)y + \mu(h)a + f(h)   \\
[y,h] &= \eta(h)a + g(h)
\end{align*}
for any $h\in H$. 

The $\omega$-Jacobi identity for triple $h_1,h_2\in H, y$ (in other words, 
the condition that $\ad h$ commute for all $h\in H$), yields
$\eta(h_1)h_2 = \eta(h_2)h_1$ for any $h_1,h_2\in H$. The latter condition implies
$\eta = 0$ provided $\dim H > 1$, i.e., $\dim L > 4$. Similarly, the $\omega$-Jacobi
identity for triple $h_1,h_2\in H, x$, yields
\begin{equation*}
\lambda(h_1)g(h_2) - \lambda(h_2)g(h_1) + \mu(h_1)h_2 - \mu(h_2)h_1 = 0
\end{equation*}
for any $h_1, h_2 \in H$. Then elementary linear-algebraic considerations imply 
that, provided $\dim H > 1$, one of the following holds:
\renewcommand{\labelenumi}{(\arabic{enumi})}
\begin{enumerate}
\item 
$\lambda \ne 0$, $\mu = t\lambda$ for some $t\in K$, and $g(h) = -th$ for $h\in Ker\,\lambda$;
\item $\lambda = \mu = 0$.
\end{enumerate}
In the second case $H$ is an abelian ideal of $L$.
In the first case, replacing $y$ by $y + ta$ and $g(h)$ by $g(h) + th$, we may assume
that $t=0$ and $g(Ker\,\lambda) = 0$.
The $\omega$-Jacobi identity for triple $y,a,h\in H$ yields $[[y,a],h] = 0$.
It is easy to see that the semisimplicity of $L$ implies that the centralizer of $H$ in $L$
coincides with $H$, so the latter equality implies $[y,a] \in H$.
Further, writing 
$$
[x,a] = \alpha x + \beta y + \gamma a + h^\prime
$$
for certain
$\alpha,\beta,\gamma\in K$, $h^\prime\in H$, and collecting terms in the 
$\omega$-Jacobi identity for triple $x,a,h\in H$, lying in $Ky$ and $H$, we get that
$\alpha = 1$ and $f(h) = -\gamma h$ for any $h\in Ker\,\lambda$.
But then $Ker\,\lambda$ is a nonzero abelian ideal of $L$.
\end{proof}

\section{Nondegenerate $\omega$}\label{section-nondeg}

In this section we treat the final, third case of Lemma \ref{main-lemma}.

\begin{lemma}\label{no-omega}
If $L$ is a finite-dimensional $\omega$-Lie algebra with nondegenerate $\omega$,
then $\dim L = 2$.
\end{lemma}

\begin{proof}
Since $\omega$ is nondegenerate, $\dim L = rank(\omega)$ is even.
First consider the case $\dim L \ge 6$. 
To treat this case, we will adopt the coordinate notation.
Though perhaps less elegant, it will make computations easier.

$L$ can be written as the direct sum of two maximal isotropic subspaces $A$ and $B$,
each of dimension $n = \frac{\dim L}{2} \ge 3$.
We may choose a basis 
$\{a_1, \dots, a_n \}$ of $A$ and a basis
$\{b_1, \dots, b_n \}$ of $B$ such that $\omega(a_i,b_i) = 1$ and $\omega(a_i,b_j) = 0$
if $i \ne j$.
Then by Lemmata \ref{xy}(iii) and \ref{2-quasi}, 
each that isotropic subspace is either abelian or almost abelian Lie subalgebra,
and it follows from the proof of Lemma \ref{2-quasi} that we may write multiplications
in them as 
$[a_i, a_j] = \alpha_j a_i - \alpha_i a_j$
and 
$[b_i, b_j] = \beta_j b_i - \beta_i b_j$
for some $\alpha_i, \beta_i \in K$.
Again, by Lemma \ref{xy}(iii), $[a_i, b_j] = \lambda_{ij} a_i + \mu_{ij} b_j$
if $i \ne j$, for some $\lambda_{ij}, \mu_{ij} \in K$.

Writing (\ref{two-basic}) for elements $b_i, b_j, b_k, a_i$, where $i,j,k$ are 
pairwise distinct, and collecting coefficients of $b_j$, we get
\begin{equation}\label{yoyo}
\omega([a_i,b_i],b_k) = \lambda_{ik}
\end{equation}
for any $i \ne k$.
Similarly, writing (\ref{two-basic}) for elements $a_i, a_j, b_i, b_k$, where $i,j,k$ are
pairwise distinct, and collecting coefficients of $a_j$, we get
$$
\omega([a_i,b_i],b_k) = \lambda_{ik} - \lambda_{jk} + \beta_k.
$$
Comparing these two equalities, we get $\lambda_{jk} = \beta_k$ for any $j \ne k$.
In a completely symmetric way, we also get 
\begin{equation}\label{yeye}
\omega([a_i,b_i],a_k) = -\mu_{ki} = \alpha_k
\end{equation}
for any $i \ne k$.

(\ref{yoyo}) and (\ref{yeye}) give all coefficients in the decomposition of 
$[a_i,b_i]$ by elements of the symplectic basis 
$\{a_1, \dots, a_n, b_1, \dots, b_n\}$, except those of $a_i$, $b_i$, so for any 
$1 \le i \le n$ we may write
$$
[a_i,b_i] = \sum_{\substack{1 \le k \le n \\ k\ne i}} (\beta_k a_k - \alpha_k b_k) + 
\lambda_i a_i + \mu_i b_i
$$
for some $\lambda_i, \mu_i \in K$.

Finally, writing (\ref{two-basic}) for elements $a_i, a_j, b_i, b_j$, where $i \ne j$,
taking into account all multiplication formulas between elements of $A$ and $B$ obtained
so far, and collecting coefficients of $a_i$ and $a_j$, we get respectively:
$\lambda_i = 2\beta_i$ and $\lambda_j = -2\beta_j$.
Consequently, $\lambda_i = \beta_i = 0$ for any $1 \le i \le n$.
Analogously, collecting coefficients of $b_i$ and $b_j$, we get $\mu_i = \alpha_i = 0$.

Therefore, $L$ is abelian. But then the $\omega$-Jacobi identity implies
that for any $3$ linearly independent elements, the values of $\omega$ on their
pairwise arguments vanish, which implies that $\omega$ vanishes, a contradiction.

The case $\dim L = 4$ requires a bit more cumbersome computations.
Note that we may assume that the ground field is algebraically closed,
as nondegeneracy of $\omega$ is obviously preserved under the ground field extension.

\smallskip

\begin{lemma}\label{3-dim}
A $4$-dimensional $\omega$-Lie algebra over an algebraically closed field
contains a $3$-dimensional subalgebra.
\end{lemma}

\begin{proof}
According to \cite[Corollary 2]{koreshkov}, any $4$-dimensional anticommutative
algebra all whose elements are nilpotent, contains a $3$-dimensional subalgebra.
Consequently, we may assume that $L$ contains a non-nilpotent element $x$.

We cannot invent anything better than proceed by boring case-by-case computations 
according to the Jordan normal form of $\ad x$ in a certain basis $\{x,y,z,t\}$ of $L$. 
Structure constants in that basis will be denoted as $C_{uv}^w$, 
the latter being the coefficient of $w$ in the decomposition of $[u,v]$,
where $u,v,w \in \{x,y,z,t\}$.

\textit{Case 1}. 
$\ad x = \begin{pmatrix}
0 & 0      & 0     & 0      \\
0 & \alpha & 0     & 0      \\
0 & 0      & \beta & 0      \\
0 & 0      & 0     & \gamma
\end{pmatrix}$.
Writing the $\omega$-Jacobi identity for triple $x,y,z$ and collecting coefficients of $t$,
we get $C_{yz}^t(\alpha + \beta - \gamma) = 0$.
If $C_{yz}^t = 0$, then $Kx \oplus Ky \oplus Kz$ forms a $3$-dimensional subalgebra.
Otherwise, $\alpha + \beta - \gamma = 0$. Repeating this argument for triples
$x,y,t$ and $x,z,t$, we get another two equalities: $\alpha - \beta + \gamma = 0$ and
$-\alpha + \beta + \gamma = 0$ respectively. The obtained homogeneous system of $3$ linear
equations in $3$ unknowns has only trivial solution, whence $\alpha = \beta = \gamma = 0$ 
and $\ad x$ is zero, a contradiction.

\textit{Case 2}.
$\ad x = \begin{pmatrix}
0 & 0      & 0      & 0      \\
0 & \alpha & 1      & 0      \\
0 & 0      & \alpha & 0      \\
0 & 0      & 0      & \beta
\end{pmatrix}$.
Writing the $\omega$-Jacobi identity for triple $x,y,t$ and collecting coefficients of $z$,
we get $C_{yt}^z \beta = 0$. If $C_{yt}^z = 0$, then $Kx \oplus Ky \oplus Kt$ forms a $3$-dimensional
subalgebra, otherwise $\beta = 0$.
Now writing the $\omega$-Jacobi identity for triple $x,y,z$ and collecting coefficients of $t$,
we get $C_{yz}^t \alpha = 0$. Since $\ad x$ is not nilpotent, $\alpha \ne 0$, hence
$C_{yz}^t = 0$, and $Kx \oplus Ky \oplus Kz$ forms a $3$-dimensional subalgebra. 

\textit{Case 3}.
$\ad x = \begin{pmatrix}
0 & 0      & 0      & 0      \\
0 & \alpha & 1      & 0      \\
0 & 0      & \alpha & 1      \\
0 & 0      & 0      & \alpha
\end{pmatrix}$.
Writing the $\omega$-Jacobi identity for triple $x,y,z$ and collecting coefficients of $t$,
we get $C_{yz}^t \alpha = 0$. Since $\alpha \ne 0$, $C_{yz}^t = 0$, and 
$Kx \oplus Ky \oplus Kz$ forms a $3$-dimensional subalgebra.

\textit{Case 4}.
$\ad x = \begin{pmatrix}
0 & 1 & 0      & 0      \\
0 & 0 & 0      & 0      \\
0 & 0 & \alpha & 0      \\
0 & 0 & 0      & \beta
\end{pmatrix}$.
Writing the $\omega$-Jacobi identity for triple $x,y,z$ and collecting coefficients of $t$,
we get $C_{yz}^t (\alpha - \beta) = 0$. If $C_{yz}^t = 0$, then $Kx \oplus Ky \oplus Kz$ forms
a $3$-dimensional subalgebra, otherwise $\alpha = \beta$. Now writing the $\omega$-Jacobi
identity for triple $x,z,t$ and collecting coefficients of $y$, we get $2C_{zt}^y \alpha = 0$.
Since $\ad x$ is not nilpotent, $\alpha \ne 0$, hence $C_{zt}^y = 0$, and 
$Kx \oplus Kz \oplus Kt$ forms a $3$-dimensional subalgebra.

\textit{Case 5}.
$\ad x = \begin{pmatrix}
0 & 1 & 0      & 0      \\
0 & 0 & 0      & 0      \\
0 & 0 & \alpha & 1      \\
0 & 0 & 0      & \alpha
\end{pmatrix}$.
Writing the $\omega$-Jacobi identity for triple $x,z,t$ and collecting coefficients of $y$,
we get $2C_{zt}^y \alpha = 0$. Since $\ad x$ is not nilpotent, $\alpha \ne 0$, hence
$C_{zt}^y = 0$, and $Kx \oplus Kz \oplus Kt$ forms a $3$-dimensional subalgebra.

\textit{Case 6}.
$\ad x = \begin{pmatrix}
0 & 1 & 0 & 0      \\
0 & 0 & 1 & 0      \\
0 & 0 & 0 & 0      \\
0 & 0 & 0 & \alpha
\end{pmatrix}$.
Finally, writing the $\omega$-Jacobi identity for triple $x,y,z$ and collecting the 
coefficients of $t$, we get $C_{yz}^t \alpha = 0$. Since $\ad x$ is not nilpotent,
$\alpha \ne 0$, hence $C_{yz}^t = 0$, and $Kx \oplus Ky \oplus Kz$ forms a $3$-dimensional subalgebra.
\end{proof}

\textit{Continuation of the proof of Lemma \ref{no-omega}.}
By Lemma \ref{3-dim}, $L$ has a $3$-dimensional subalgebra $M$.
Clearly, $\dim Ker\,\omega|_M$ is equal to $1$ or $3$. 
In the latter case $M$ is a Lie algebra, and $\omega$ is necessarily 
degenerate on the whole $L$. Hence $M$ is a $3$-dimensional $\omega$-Lie algebra
which is not a Lie algebra. 

All isomorphism classes of $3$-dimensional $\omega$-Lie algebras are listed in 
\cite[Theorem 2.1]{nurowski}\footnote[2]{
In \cite{nurowski}, algebras are classified over the field of real numbers,
but the classification readily extends to any ground field of characteristic $\ne 2$.
Over an algebraically closed field, types $(VIII)_a$ and $(IX)_a$ are isomorphic,
and over any field all parametric types are isomorphic for parameters $a$ and $-a$.
Note also that the definition of the $\omega$-Jacobi identity adopted here
differs from those in \cite{nurowski} by the sign of $\omega$.
}, 
and following the scheme of \S \ref{sec-der}, our task 
amounts to finding $(\alpha, \lambda)$-derivations of these algebras.
$\lambda$ can be found from (\ref{mult-1}), which in all cases amounts
to a linear system of $3$ equations in $3$ unknowns 
(values of $\lambda$ on the basic elements), having either single or a $1$-parametric 
solution.

We used a primitive GAP code, described in Appendix, to find that for any
$(\alpha,\lambda)$-derivation of any $3$-dimensional $\omega$-Lie algebra $M$, $\alpha$ 
vanishes on $Ker\,\omega|_M$. Consequently, for the appropriate $4$-dimensional
$\omega$-Lie algebra $L$, $Ker\,\omega \supseteq Ker\,\omega|_M \ne 0$, i.e., $\omega$
is degenerate.
\end{proof} 

\section{Main theorems}\label{sect-main}

To summarize results of Lemmata \ref{main-lemma}, \ref{der-lie-algs}, \ref{codim-2}
and \ref{no-omega}:

\begin{theorem}\label{main}
Let $L$ be a finite-dimensional $\omega$-Lie algebra which is not a Lie algebra. 
Then one of the following holds:
\begin{enumerate}
\item $\dim L = 3$.
\item 
$L$ has a Lie subalgebra of codimension $1$ whose structure is described by
Lemma {\rm \ref{der-lie-algs}(iii)}.
\item 
$Ker\,\omega$ is an almost abelian Lie algebra of codimension $2$ in $L$ with the abelian
part acting nilpotently on $L$.
\end{enumerate}
In all the cases, $L$ has an abelian subalgebra of codimension $\le 3$.
\end{theorem}

To summarize further results of Corollary \ref{simple-codim-1} and 
Lemma \ref{simple-codim2}:

\begin{theorem}\label{th-simple}
A finite-dimensional semisimple $\omega$-Lie algebra is either a Lie algebra, or has dimension
$\le 4$.
\end{theorem}

Thus, the structure of $\omega$-Lie algebras beyond dimension $3$ turns out to be quite 
``degenerate'', and, in a sense, all the interesting cases are already presented 
in \cite{nurowski}.

In principle, it is possible to enumerate on computer all isomorphism clas\-ses of 
$4$-di\-men\-si\-o\-nal $\omega$-Lie algebras, and, among them, of all simple algebras, 
but we will not venture into this: as noted at the end of \S \ref{sec-der}, there 
are a lot of them, sometimes with quite cumbersome multiplication tables.

Here is just one example of a $4$-dimensional simple $\omega$-Lie algebra, obtained
via appropriate $(\alpha,\lambda)$-derivation from the algebra of type $(IV)_T$ in the
Nurowski's list:
\begin{equation*}
[e_1,e_2] = e_2, \> [e_1,e_3] = e_3, \> [e_2,e_3] = e_1, \> [e_1,e_4] = -e_3 + 2e_4, \>
[e_2,e_4] = e_1, \> [e_3,e_4] = 0,
\end{equation*}
and the only nonzero values of $\omega$ on the pairs of elements from the basis are:
$$
\omega(e_2,e_3) = \omega(e_2,e_4) = 2 .
$$

\section{Identities}\label{ident}

In this section we address the following natural question: what identities are satisfied
by $\omega$-Lie algebras? 
Note that one should distinguish identities of $\omega$-Lie algebras as ordinary 
algebras, i.e., in the signature consisting of one binary operation which is an algebra 
multiplication, and as $\omega$-algebras, i.e., in the signature consisting
of one binary operation representing multiplication, and one binary operation with 
values in the ground field, representing the form $\omega$. Let us call identities in
the latter sense \textit{$\omega$-identities}. Clearly, the $\omega$-Jacobi identity is an 
$\omega$-identity, and we are primarily interested in (ordinary) identities in the former
sense.

One of the fruitful methods to study identities of 
algebras is to superize the situation, and consider the Grassmann envelopes of 
corresponding superalgebras. But to be able to apply this method, 
the class of algebras under consideration should be closed under tensoring with an 
associative commutative algebra. It is easy to see that this is not so for $\omega$-Lie 
algebras. To this end, we enlarge the definition of $\omega$-Lie algebras by allowing 
$\omega$ to take values in the centroid of the algebra, instead of merely in the ground
field:

\begin{definition}
An anticommutative algebra $L$ is called an \textit{extended $\omega$-Lie algebra}
if there is a bilinear map $\omega: L \times L \to Cent(L)$ such that the $\omega$-Jacobi identity
(\ref{jacobi}) holds.
\end{definition}

Here $Cent(L)$ denotes the centroid of an algebra $L$.

For any algebra $L$ and associative commutative algebra $A$, we have an obvious inclusion
$$
Cent(L) \otimes A \subseteq Cent(L\otimes A) .
$$
If $L$ is an extended $\omega$-Lie algebra, then, defining a bilinear map
\begin{align}
\Omega: (L\otimes A) \times (L\otimes A) &\to Cent(L) \otimes A \notag \\
        (x\otimes a, y\otimes b) &\mapsto \omega(x,y) \otimes ab , \label{Omega}
\end{align}
where $x,y\in L$, $a,b\in A$, we see that $L\otimes A$ becomes an extended 
$\Omega$-Lie algebra.

It is clear that the class of $\omega$-Lie algebras and the class of extended $\omega$-Lie algebras
satisfy the same identities and $\omega$-identities.

The notion of (extended) $\omega$-Lie algebra can also be generalized to the super case.
For a superalgebra $L$, let $Cent(L)$ denote a supercentroid of $L$
(which is a generalization of the ordinary centroid).

\begin{definition}
A super-anticommutative superalgebra $L = L_0 \oplus L_1$ is called an 
\textit{extended $\omega$-Lie superalgebra}
if there is a super-skew-symmetric bilinear map $\omega: L \times L \to Cent(L)$ such that
\begin{multline}\label{super}
  (-1)^{\deg x \deg z}[x,[y,z]]
+ (-1)^{\deg z \deg y}[z,[x,y]]
+ (-1)^{\deg y \deg x}[y,[z,x]]  \\
+ (-1)^{\deg x \deg z}\omega(y,z)x
+ (-1)^{\deg z \deg y}\omega(x,y)z 
+ (-1)^{\deg y \deg x}\omega(z,x)y = 0
\end{multline}
holds for any homogeneous elements $x,y,z\in L$.
\end{definition}

As in the ordinary case, one easily sees that if $L = L_0 \oplus L_1$ is an extended 
$\omega$-Lie superalgebra, and $A = A_0 \oplus A_1$ is a commutative superalgebra, then
the algebra $(L_0 \otimes A_0) \oplus (L_1 \otimes A_1)$ is an extended $\Omega$-Lie
algebra for $\Omega$ defined by formula (\ref{Omega}), for the respective homogeneous
elements $x,y,a,b$. In particular, the Grassmann envelope of an extended $\omega$-Lie 
superalgebra is an $\Omega$-Lie algebra for a suitable $\Omega$.

Now, again, like in the case of ordinary algebras, we have the well-known correspondence
between $\omega$-identities of an $\omega$-Lie superalgebra and $\omega$-identities of
its Grassmann envelope (realized, on the level of multilinear identities, by injecting
appropriate signs at appropriate places). This provides a compact method to write
identities and $\omega$-identities of (extended) $\omega$-Lie algebras.

For example, the $\omega$-Jacobi superidentity (\ref{super}), written for triples 
$x,x,x$ and $x,x,[x,x]$, implies respectively
\begin{gather}
[[x,x],x] = \omega(x,x)x                  \label{32} \\  
2\omega([x,x],x)x + \omega(x,x)[x,x] = 0  \label{33} 
\end{gather}
for any (odd) $x\in L$. It is possible to show that the full linearization of the
$\omega$-identity (\ref{33}) is equivalent to the super-analog of the identity (\ref{two-basic}).

In its turn, (\ref{32}) implies the identity 
$$
[[[[x,x],x],x],x] + [[[x,x],x],[x,x]] = 0
$$
Linearizing the latter identity and taking its ``ordinary'' part, one arrives at the
following identity of degree $5$ satisfied by all $\omega$-Lie algebras:
$$
\sum_{\sigma\in S_5} (-1)^\sigma 
\Big([[[[x_{\sigma(1)},x_{\sigma(2)}],x_{\sigma(3)}],x_{\sigma(4)}],x_{\sigma(5)}] +
[[[x_{\sigma(1)},x_{\sigma(2)}],x_{\sigma(3)}],[x_{\sigma(4)},x_{\sigma(5)}]]\Big) = 0 .
$$

This is identity of the smallest possible degree:

\begin{proposition}\label{prop-ident}
The minimal degree of identity which is satisfied by all $\omega$-Lie algebras is $5$.
\end{proposition}

\begin{proof}
It is sufficient to show that no identity of degree $\le 4$ is satisfied by all
$\omega$-Lie algebras.
Irreducible identities of degree $\le 4$ of anticommutative algebras were described in 
\cite[\S 2]{kuzmin} and \cite[Theorem 3]{kw}. According to these results, every
anticommutative algebra with multiplication $\liebrack$, satisfying an identity of degree
$\le 4$, satisfies one of the following identities:
\begin{gather}
[[[y,x],x],x] = 0  \label{engel} \\ 
  \alpha[[x,y],[x,z]] + \beta\Big([[[x,y],x],z] - [[[x,z],x],y]\Big) \label{abg} \\ 
+ \gamma\Big([[[x,y],z],x] - [[[x,z],y],x]\Big) +
(\beta + \gamma)[[[y,z],x],x] = 0  \notag  \\
J(x,y,[x,y]) = 0  \label{bin} \\
[J(x,y,z),t] - [J(t,x,y),z] + [J(z,t,x),y] - [J(y,z,t),x] = 0 , \label{4}
\end{gather}
where $J(x,y,z) = [[x,y],z] + [[z,x],y] + [[y,z],x]$ is the Jacobian, and 
$\alpha, \beta, \gamma$ in (\ref{abg}) are some fixed elements of the ground field.

Identities (\ref{engel}) and (\ref{abg}) are not satisfied even in the narrower class
of Lie algebras: (\ref{engel}) is the $3$rd Engel condition, and (\ref{abg}) is not
satisfied, for example, in the free $4$-generated Lie algebra.

Identity (\ref{bin}) defines binary-Lie algebras. Being coupled with the 
$\omega$-Jacobi identity (\ref{jacobi}), it implies
$$
\omega(x,y)[x,y] = \omega([x,y],y)x - \omega([x,y],x)y
$$
The latter condition is violated, for example, for most of the 
$3$-dimensional algebras in Nurowski's list \cite{nurowski}.

Similarly, (\ref{4}) together with (\ref{jacobi}) implies
$$
  \omega(z,t)[x,y]
+ \omega(t,y)[x,z]
+ \omega(y,z)[x,t]
+ \omega(x,t)[y,z]
+ \omega(z,x)[y,t]
+ \omega(x,y)[z,t]
= 0 .
$$
In view of Lemma \ref{4-var}, for $\omega$-Lie algebras of dimension $>3$ this is equivalent to 
$$
\omega([x,y],z) + \omega([z,x],y) + \omega([y,z],x) = 0 .
$$
The last condition is not fulfilled, for example, for $4$-dimensional $\omega$-Lie 
algebras obtained by extending $sl(2)$ by its $(\alpha,0)$-derivations, 
described at the end of \S \ref{sec-der}.
\end{proof}

\section{Further questions}\label{section-quest}

\begin{question}
What happens in characteristics $2$ and $3$?
\end{question}

In the case of characteristic $2$ an entirely different approach 
(and, perhaps, a different definition of an $\omega$-Lie algebra)
would be needed. 
On the contrary, the assumption that the characteristic of the ground field is
different from $3$, was used only twice,
in the key Lemma \ref{xy} and when performing calculations with $3$-dimensional
$\omega$-Lie algebras described at the end of \S \ref{section-nondeg}.
(Also, in characteristic $3$ the $\omega$-identity (\ref{32}) no longer follows
from the $\omega$-Jacobi superidentity, and one needs to include it in the definition
of $\omega$-Lie superalgebra). 
More accurate reasonings could show that a statement similar to Lemma \ref{xy}
still holds in characteristic $3$, with stronger conditions
on the dimension of $rank\,\omega$ (basically, shifted to 2). 

%
So, probably the case of characteristic $3$ could be treated along the lines
of the present paper.

\begin{question}\label{quest-deform}
Which Lie algebras are deformed into non-Lie $\omega$-Lie algebras?
\end{question}

As we learned from Rutwig Campoamor-Stursberg, there was a hope to get
some physically meaningful contractions of $\omega$-Lie algebras into simple Lie algebras.
From Theorem \ref{main} it is clear that in dimension $>3$ this is impossible --
contracted Lie algebras should be not less degenerate than $\omega$-Lie algebras,
close to abelian ones.

Nevertheless, one can still ask which Lie algebras could arise as such contractions,
which is, essentially, equivalent to the question: which Lie algebras
could be deformed into $\omega$-Lie algebras?

Let us try to develop a rudimentary deformation theory of $\omega$-Lie algebras, 
following the 
standard nowadays format suggested by Gerstenhaber in \cite{gerstenhaber}.
A deformation of an $\omega$-Lie algebra $L$ (which can be just a Lie algebra with
$\omega = 0$) is an $\omega_t$-Lie algebra $L_t$
defined over a power series ring $K[[t]]$ whose multiplication $\liebrack_t$ 
and form $\omega_t$ satisfy the conditions
\begin{gather*}
[x,y]_t = [x,y] + \varphi_1(x,y) t + \varphi_2(x,y) t^2 + \dots   \\
\omega_t (x,y) = \omega(x,y) + \omega_1(x,y) t + \omega_2(x,y) t^2 + \dots
\end{gather*}
for certain bilinear maps $\varphi_n: L \times L \to L$ and $\omega_n: L \times L \to K$.

Anticommutativity of $\liebrack_t$ and skew-symmetricity of $\omega_t$ imply that each 
$\varphi_n$ is anticommutative and each $\omega_n$ is skew-symmetric.
The $\omega$-Jacobi identity for $L_t$ is equivalent to:
$$
d\,\varphi_n(x,y,z) + 
\sum_{\substack{i+j= n \\ i,j > 0}} [\varphi_i, \varphi_j] (x,y,z) =
\omega_n(x,y)z + \omega_n(z,x)y + \omega_n(y,z)x 
$$
for each $n=1,2, \dots$ and $x,y,z\in L$, 
where $d$ is the second-order Chevalley-Eilenberg differential in the Lie algebra
(= $\omega$-Lie algebra) 
cohomology, and $\liebrack$ is the usual Massey product of 2-cochains.

The first of these equalities ($n=1$) reads:
\begin{multline}\label{infi}
\varphi_1([x,y],z) + \varphi_1([z,x],y) + \varphi_1([y,z],x) +
[\varphi_1(x,y),z] + [\varphi_1(z,x),y] + [\varphi_1(y,z),x] \\ = 
\omega_1(x,y)z + \omega_1(z,x)y + \omega_1(y,z)x .
\end{multline}

Thus, the question reduces to: which Lie algebras admit infinitesimal deformations 
(\ref{infi}) with nontrivial $\omega_1$?

\begin{question}
Are there ``interesting'' examples of infinite-dimensional $\omega$-Lie algebras
and of $\omega$-Lie superalgebras?
\end{question}

Could it be that in the super or, more general, color case, new phenomena will arise
making the structure theory more colorful, for example, allowing the existence of some 
interesting simple objects?

\begin{question}
What would be analogs of $\omega$-Lie algebras for other classes of algebras?
\end{question}

By analogy with the $\omega$-Jacobi identity, one may to alter the associative identity 
as follows:
\begin{equation}\label{ass-bad}
(xy)z - x(yz) = \omega(x,y)z - \omega(y,z)x .
\end{equation}
One of the main features of associative algebras is that they are Lie-admissible, and one
may wish to preserve this relationship for their $\omega$-variants: namely, that if $A$ 
is an $\omega$-associative algebra, then its ``minus'' algebra with multiplication
$[x,y] = xy - yx$ for $x,y\in A$, would be $\omega$-Lie. An easy calculation shows that the 
``minus'' algebra of an algebra satisfying the identity (\ref{ass-bad}) is Lie, and not 
just $\omega$-Lie. That
indicates that (\ref{ass-bad}) is probably not an adequate definition of 
$\omega$-associativity, and one may wish to alter it further as follows:
\begin{equation}\label{ass-good}
(xy)z - x(yz) = \omega_1(x,y)z - \omega_2(y,z)x 
\end{equation}
for some two bilinear forms $\omega_1, \omega_2: L \times L \to K$
(one may argue that, unlike in the Lie case, to reflect the difference of order in 
multiplication in two terms on the left-hand side, two different bilinear forms 
$\omega_1$ and $\omega_2$ are required).
A ``minus'' algebra
of an algebra satisfying the latter identity is indeed an $\omega$-Lie algebra, with
$$
\omega(x,y) = (\omega_1 - \omega_2)(x,y) - (\omega_1 - \omega_2)(y,x) .
$$
Is (\ref{ass-good}) a ``correct'' definition of an $\omega$-associative algebra?
Does it lead to anything interesting?

Similarly, what would be ``correct'' definitions for $\omega$-Leibniz
algebras, $\omega$-Novikov algebras, $\omega$-left-(or right-)symmetric algebras, etc.?

\section*{Acknowledgements}

Thanks are due to Rutwig Campoamor-Stursberg, Zhiqi Chen and Friedrich Wagemann for 
interesting discussions, and to Dimitry Leites for suggestions which improved the 
presentation. A previous version of this manuscript was rejected by 
\textit{Journal of Algebra}, but I owe to the anonymous referee of that submission a 
great deal of useful comments: in particular, he urged formulating
Theorem \ref{th-simple} in \S \ref{sect-main} explicitly, and most of the 
material in \S \ref{ident} of the present version stems from his suggestions.

\appendix
\section*{Appendix}

Here we describe a simple-minded GAP code, available at \newline
\texttt{http://justpasha.org/math/alder.gap}, for calculating 
$(\alpha,\lambda)$-de\-ri\-va\-t\-i\-ons of an $\omega$-Lie algebra,
mentioned in \S\S \ref{sec-der} and \ref{section-nondeg}.

Let $L$ be an anticommutative algebra with the basis $\{e_1, \dots, e_n\}$ defined
over a field $K$, and with multiplication table 
$$
[e_i,e_j] = \sum_{k=1}^n C_{ij}^k e_k ,
$$ 
and let $D$ be an 
$(\alpha,\lambda)$-derivation of $L$. Writing $D(e_i) = \sum_{j=1}^n d_{ij}e_j$,
$\lambda(e_i) = \lambda_i$, and $\alpha(e_i) = \alpha_i$
for certain $d_{ij}, \lambda_i, \alpha_i \in K$,
the condition (\ref{pre-der}), written
for each pair of basic elements, is equivalent to the system of $\frac{n^2(n-1)}{2}$
linear equations in $n^2 + n$ unknowns $d_{ij}$ and $\alpha_i$:
\begin{equation}\label{sys}
\sum_{k=1}^n C_{ij}^k d_{kl} - \sum_{k=1}^n C_{kj}^l d_{ik} + \sum_{k=1}^n C_{ki}^l d_{jk} 
- \lambda_j d_{il} + \lambda_i d_{jl} - \delta_{il} \alpha_j + \delta_{jl} \alpha_i = 0
\end{equation}
for $1 \le i < j \le n$, $1 \le l \le n$ ($\delta_{ij}$ is the Kronecker symbol).

If, additionally, $L$ is an $\omega$-Lie algebra, $\lambda_i$ can be found from 
(\ref{mult-1}), which is equivalent to the system of $\frac{n(n-1)}{2}$ 
linear equations in $n$ unknowns:
$$
\sum_{k=1}^n C_{ij}^k \lambda_k = \omega(e_i,e_j) 
$$
for $1 \le i < j \le n$.

So, taking the structure constants of an algebra, as well $\lambda$ as an input
(possibly involving parameters), we just solve the linear homogeneous system (\ref{sys}).

As GAP (version 4.4.12 as of time of this writing) does not support transcendental
field extensions -- which would be the natural way to work with parameters --
we are cheating by using cyclotomic fields instead.
However, this cheating could be made rigorous by picking a cyclotomic extension of a prime
degree (of course, any other field extension by an irreducible polynomial will do) larger than
the highest possible power of a parameter involved in computation.
For example, if we deal with $3$-dimensional algebras, the system (\ref{sys}) is 
of size $9 \times 12$, so if a parameter enters linearly into the initial data, 
any of its powers occuring in the solution of the system does not exceed $9$, so the cyclotomic 
extension of order $11$ will be enough.

\end{document}